\documentclass[12pt]{article}
\usepackage{amssymb,amsmath,epsfig,amscd,eucal,psfrag,amsthm}

\pdfoutput=1

\newtheorem{theorem}{Theorem}

\newtheorem{lemma}[theorem]{Lemma}
\newtheorem{proposition}[theorem]{Proposition}

\newtheorem{corollary}[theorem]{Corollary}

\newtheorem{question}[theorem]{Question}

\theoremstyle{definition}
\newtheorem{definition}[theorem]{Definition}

\theoremstyle{remark}
\newtheorem{remark}[theorem]{Remark}

\newcommand{\uu}{\underline{u}}
\newcommand{\vv}{\underline{v}}
\newcommand{\w}{\underline{w}}

\input xy
\xyoption{all}

\title{One-ended subgroups of graphs of free groups with cyclic edge groups}
\author{Henry Wilton}

\begin{document}

\maketitle

\begin{abstract}
Consider a one-ended word-hyperbolic group.  If it is the fundamental group of a graph of free groups with cyclic edge groups then either it is the fundamental group of a surface or it contains a finitely generated one-ended subgroup of infinite index.  As a corollary, the same holds for limit groups.  We also obtain a characterisation of surfaces with boundary among free groups equipped with peripheral structures.
\end{abstract}

Free subgroups of hyperbolic groups are abundant, and many successful techniques have been developed to find them.  It is necessarily much harder to find one-ended subgroups: any one-ended group has at most finitely many conjugacy classes of embeddings into a fixed hyperbolic group \cite[5.3.C']{gromov_hyperbolic_1987}.   One case of longstanding interest in topology is the problem of finding a \emph{surface subgroup}, by which we mean a subgroup isomorphic to the fundamental group of a closed surface of non-positive curvature.  The following famous question is often attributed to Gromov \cite{bestvina_questions_????,bridson_problems_????}.

\begin{question}\label{q: Surface subgroups}
Does every one-ended hyperbolic group have a surface subgroup?
\end{question}

The closest thing to this question that we were able to find in the literature is as follows \cite[page 144]{gromov_hyperbolic_1987}.

\begin{quote}
`[O]ne may suspect that there exist word hyperbolic groups $\Gamma$ with arbitrarily large $\mathrm{dim}~\partial\Gamma$ (here, large is $\geq 1$) where every proper subgroup is free.'
\end{quote}
Setting aside the question of whether every hyperbolic group has a proper subgroup of finite index, which is a notoriously difficult question in its own right, this raises the following natural counterpart to Question \ref{q: Surface subgroups}.

\begin{question}\label{q: One-ended subgroups}
Does every one-ended hyperbolic group that is not virtually a surface group contain a finitely generated one-ended subgroup of infinite index?
\end{question}

Although Question \ref{q: One-ended subgroups} seems substantially weaker than Question \ref{q: Surface subgroups}, in the motivating case of the fundamental group of a closed hyperbolic 3-manifold (recently resolved in the affirmative by Kahn and Markovic \cite{kahn_immersing_2009}), the two questions are equivalent.  Indeed, since 3-manifold groups are coherent \cite{scott_finitely_1973}, any finitely generated, infinite-index subgroup $H$ of a 3-manifold group is the fundamental group of a compact 3-manifold $N$ with non-empty boundary.  If $N$ is aspherical then $\partial N$ has non-positive Euler characteristic.  If, in addition, $H$ is one-ended then Dehn's Lemma implies that $\partial N$ is incompressible, and so $H$, and hence $\pi_1N$, has a surface subgroup.

Despite the recent result of Kahn and Markovic, little is known about certain seemingly very simple classes of hyperbolic groups.  For instance, let $\Gamma$ be the fundamental group of a graph of free groups with infinite cyclic edge groups.  Calegari proved that $\Gamma$ has a surface subgroup if $H_2(\Gamma;\mathbb{Q})$ is non-trivial \cite{calegari_surface_2008}. Further sufficient conditions were found in \cite{gordon_surface_2010} and \cite{kim_polygonal_2009}.  Kim and Oum answered Question \ref{q: Surface subgroups} when $\Gamma$ is the double of a rank-two free group \cite{kim_hyperbolic_2010}.

By the Combination Theorem \cite{bestvina_combination_1992}, $\Gamma$ is hyperbolic if and only if it does not contain a Baumslag--Solitar subgroup, and the existence of such a subgroup can be verified from a given graph-of-groups decomposition.  One-endedness can also be easily characterised in this case: $\Gamma$ is one-ended if and only if every vertex group of the graph of groups is freely indecomposable relative to its edge groups; see Theorem \ref{t: Variant of Shenitzer's Theorem} below for details.

The main theorem of this paper resolves Question \ref{q: One-ended subgroups} for such $\Gamma$.

\begin{theorem}\label{t: Main theorem}
Let $\Gamma$ be the fundamental group of a graph of free groups with cyclic edge groups. If $\Gamma$ is hyperbolic and one-ended then either $\Gamma$ is the fundamental group of a closed surface or $\Gamma$ has a finitely generated subgroup of infinite index that is one-ended.
\end{theorem}

It is obvious that one can reduce to the case in which the given splitting of $\Gamma$ has only one edge---that is, to the case in which $\Gamma$ is an amalgamated free product or HNN extension of free groups.  However, this observation does not seem to be particularly useful.  In fact, the proof uses a cyclic splitting of $\Gamma$ that may be finer, namely the JSJ decomposition.

Of course, graphs of free groups with cyclic edge groups are not representative of hyperbolic groups in general, so Theorem \ref{t: Main theorem} falls far short of resolving Question \ref{q: One-ended subgroups}, but it does place heavy restrictions on the nature of any positive example.  If $\Gamma$ is any hyperbolic group that splits over a virtually cyclic subgroup then Question \ref{q: One-ended subgroups} is trivial for $\Gamma$ unless $\Gamma$ is the fundamental group of a graph of virtually free groups with virtually cyclic edge groups.  Such groups are virtually torsion-free \cite[Theorem 4.19]{wise_subgroup_2000}, and so the main theorem has the following corollary.

\begin{corollary}\label{c: Grandiose corollary}
Suppose $\Gamma$ is a one-ended hyperbolic group that splits non-trivially over a virtually cyclic subgroup.  Then either $\Gamma$ has a finitely generated one-ended subgroup of infinite index or $\Gamma$ is virtually the fundamental group of a closed surface.
\end{corollary}

Limit groups, otherwise known as finitely generated fully residually free groups, play a central role in the study of algebraic geometry and logic over free groups \cite{sela_diophantine_2001,kharlampovich_irreducible_1998}.  They are not necessarily hyperbolic, but every limit group is a toral relatively hyperbolic group \cite{alibegovic_combination_2005,dahmani_combination_2003}.

\begin{corollary}\label{c: Limit groups}
If $\Gamma$ is a one-ended limit group  then either $\Gamma$ has a finitely generated one-ended subgroup of infinite index or $\Gamma$ is the fundamental group of a closed surface.
\end{corollary}
\begin{proof}
A limit group is hyperbolic if and only if it does not have a subgroup isomorphic to $\mathbb{Z}^2$ \cite[Corollary 4.4]{sela_diophantine_2001}, and any virtually abelian limit group is abelian, so we may reduce to the hyperbolic case.  Every non-abelian limit group splits over a cyclic subgroup \cite[Theorem 3.9]{sela_diophantine_2001}, so the result follows from Corollary \ref{c: Grandiose corollary}.
\end{proof}

Further motivation for Theorem \ref{t: Main theorem} is provided by the class of \emph{special} groups, introduced by Haglund and Wise \cite{haglund_special_2008}, of which graphs of free groups with cyclic edge groups are examples \cite{tim_hsu_cubulating_2010}.  It should be possible to generalise Theorem \ref{t: Main theorem} to the class of special groups; the main technical obstruction is the absence of a suitable JSJ decomposition.

The ingredients of the proof of Theorem \ref{t: Main theorem} include Bowditch's JSJ decomposition for hyperbolic groups and a criterion for detecting free splittings of graphs of groups with cyclic edge groups in terms of their vertex groups (Theorem \ref{t: Variant of Shenitzer's Theorem}).  The heart of the proof is a `Local Theorem' about the rigid vertices of the JSJ decomposition (Theorem \ref{t: Local theorem} below).  To state the Local Theorem, we need to introduce peripheral structures on free groups.

A \emph{multiword} in $F$ is a subset $\w\subseteq F\smallsetminus 1$.  A set of pairwise non-conjugate maximal cyclic subgroups of $F$ is called a \emph{peripheral structure} on $F$.   Any multiword $\w$ defines a peripheral structure $[\w]$.  We will consider pairs $(F,[\w])$, where $F$ is a free group and $[\w]$ is a peripheral structure.  Note that any vertex group of a graph of groups with cyclic edge groups is naturally equipped with a peripheral structure induced by the incident edges.

A peripheral structure $[\w]$ on $F$ induces a natural pullback peripheral structure $[\hat{\w}]$ on any subgroup $\widehat{F}\subseteq F$.  We give a topological definition.

\begin{definition}
Realise $F$ as the fundamental group of a handlebody $X$, and $\w$ as an embedded 1-dimensional submanifold $N\subseteq X$.  If $\widehat{F}\subseteq F$ is a subgroup, let $\widehat{X}$ be a corresponding covering space of $X$ and let $\widehat{N}$ be the preimage of $N$ in $\widehat{X}$.  The pullback of $[\w]$ to $\widehat{F}$, which we denote by $[\hat{\w}]$, is determined by those non-trivial conjugacy classes of $\widehat{F}$ that are determined by components of $\widehat{N}$.  
\end{definition}

\begin{remark}
If $\w$ is finite and $\widehat{F}$ is finitely generated then $\hat{\w}$ is finite .
\end{remark}

A pair $(F,[\w])$ is said to be \emph{freely indecomposable} or \emph{one-ended} if the elements of $[\w]$ are elliptic in every free splitting of $F$, and \emph{rigid} if they are elliptic in every cyclic splitting of $F$.  A pair $(F,[\w])$ is said to be a \emph{surface} if there is an isomorphism $F\cong\pi_1\Sigma$ for $\Sigma$ a compact surface that identifies $[\w]$ with the conjugacy classes of cyclic subgroups corresponding to $\partial\Sigma$. An important special case is when $\Sigma$ is a \emph{thrice-punctured sphere}; this is the unique case in which $(F,[\w])$ is both a surface and rigid.

We can now state the Local Theorem.

\begin{theorem}\label{t: Local theorem}
Suppose $(F,[\w])$ is rigid and not a thrice-punctured sphere. For any clean finite-index subgroup $\widehat{F}\subseteq F$, for any $\hat{w}_i\in\hat{\w}$, the pair $(\widehat{F},[\hat{\w}\smallsetminus \{\hat{w}_i\}])$ is freely indecomposable.
\end{theorem}

See Definition \ref{d: Clean} below for the definition of a clean subgroup.  It follows from Marshall Hall's Theorem \cite{hall_jr_subgroups_1949} that there are many clean subgroups of finite index.

Finally, let us consider the extent to which the results of this paper have a bearing on Question \ref{q: Surface subgroups}.  Our techniques provide a new characterisation of surfaces with boundary $(\pi_1\Sigma,\partial\Sigma)$ among all pairs $(F,[\w])$: they are precisely the pairs in which every non-abelian subgroup of infinite index, equipped with the induced peripheral structure, is freely indecomposable.  To develop this idea, we introduce the natural partial order on commensurability classes of subgroups of free groups with peripheral structures.  A peripheral structure $[\uu]$ on $\widehat{F}$ is \emph{compatible} with $[\w]$ if $[\uu]\subseteq[\hat{\w}]$.  We define a preorder on subgroups of $F$ equipped with compatible peripheral structures as follows.

\begin{definition}
Given $(F,[\w])$, let $(H,[\uu])$ and $(K,[\vv])$ be subgroups equipped with peripheral structures compatible with $[\w]$. Let $[\hat{\uu}]$ be the peripheral structure on $H\cap K$ induced by $[\uu]$.  Write
\[
(H,[\uu])\leq (K,[\vv])
\]
if:
\begin{enumerate}
\item $|H:H\cap K|<\infty$; and furthermore, 
\item if $|K:H\cap K|<\infty$ also then $[\hat{\uu}]$ is compatible with $[\vv]$.
\end{enumerate}
This is a preorder on subgroups equipped with compatible peripheral structures; the induced equivalence relation is called \emph{commensurability}, and the preorder $\leq$ descends to a partial order on commensurability classes.
\end{definition}

\begin{definition}
Let $\mathcal{P}(\w)$ be the poset of commensurability classes of pairs $(H,[\uu])$ such that:
\begin{enumerate}
\item $H$ is non-abelian and a finitely generated subgroup of $F$;
\item $[\uu]$ is compatible with $[\w]$;
\item the pair $(H,[\uu])$ is freely indecomposable.
\end{enumerate}
\end{definition}

\begin{corollary}\label{c: Minimal commensurability classes}
Suppose that $(F,[\w])$ is freely indecomposable.  Let $(H,[\uu])$ be a pair that represents a commensurability class in $\mathcal{P}(\w)$.  The commensurability class represented by $(H,[\uu])$ is minimal in $\mathcal{P}(\w)$ if and only if
\[
(H,[\uu])\cong(\pi_1\Sigma,\partial\Sigma)
\]
for some compact surface with boundary $\Sigma$.
\end{corollary}

This raises the hope of applying Zorn's Lemma to find surface subgroups.

\begin{question}\label{q: Lower bound}
Does every chain in $\mathcal{P}(\w)$ have a lower bound?
\end{question}

An affirmative answer to Question \ref{q: Lower bound} would come very close to resolving Question \ref{q: Surface subgroups} in the case of graphs of free groups.  (There are also some mild compatibility conditions needed on the surfaces constructed.)  This would indicate that surfaces with boundary in $[\w]$ are fairly abundant in $F$.  A negative answer would indicate that such surfaces are very special indeed. 

\subsection*{Acknowledgements}

Thanks to Sang-hyun Kim for pointing me towards Question \ref{q: One-ended subgroups}.  Thanks to Cameron Gordon for teaching me everything I know about handlebodies and Whitehead's algorithm, and also for a useful conversation about Lemma \ref{l: Separating pair of edges}. Thanks to Chris Cashen and Larsen Louder for some instructive examples.

\section{Covering theory of graphs of spaces}\label{s: Covering theory}

In the proof of Theorem \ref{t: Main theorem}, we will need to use the covering theory of graphs of groups.   Such a theory was developed from an algebraic point of view by Bass \cite{bass_covering_1993}, but instead we will use the point of view of graphs of spaces, following Scott and Wall \cite{scott_topological_1979}.

The data for a graph of spaces $X$ is as follows.  We are given a graph $\Xi$, for each vertex $v$ of $\Xi$ a connected CW-complex $X_v$, and for each edge $e$ of $\Xi$ a connected CW-complex $X_e$.  If an edge $e$ adjoins vertices $v_{\pm}$, we are given corresponding attaching maps $\partial^{\pm}_e:X_e\to X_{v_\pm}$ .  These attaching maps $\partial^{\pm}_e$ are required to be $\pi_1$-injections.

The \emph{geometric realisation} of $X$ is the space
\[
\left(\coprod_{v\in V(\Xi)}X_v\sqcup\coprod_{e\in E(\Xi)} (X_e\times [-1,+1])\right)/\sim
\]
where the relation $\sim$ identifies $(x,\pm 1)\in X_e\times [-1,+1]$ with $\partial^\pm_e(x)\in X_{v_\pm}$, for each edge $e\in E(\Xi)$ and each $x\in X_e$.  We will usually abuse notation and denote the geometric realisation of $X$ by $X$.

From this topological point of view, a graph of groups with fundamental group $\Gamma$ is simply an Eilenberg--Mac~Lane space $X$ for $\Gamma$ with the structure of a graph of spaces. 

The key component of the covering theory for graphs of spaces is the definition of an elevation, which was first introduced by Wise (see, for instance, \cite{wise_subgroup_2000}).  The covering theory of graphs of spaces was further developed in \cite{wilton_elementarily_2007,wilton_halls_2008}, to which the reader is referred for proofs of some of the statements below.

Let $X$ be a graph of spaces, with underlying graph $\Xi$, vertex spaces $X_v$, edge spaces $X_e$, and attaching maps $\partial^{\pm}_e:X_e\to X_v$.  A covering space $\widehat{X}$ of $X$ is naturally endowed with the structure of a graph of spaces.  The connected components of the preimages of the vertex spaces of $X$ form the vertex spaces of $\widehat{X}$, and likewise the connected components of the preimages of the edge cylinders of $X$ form the edge cylinders of $\widehat{X}$.  The underlying graph of $\widehat{X}$, denoted by $\widehat{\Xi}$, can be recovered by collapsing the vertex spaces of $\widehat{X}$ to points and the edge spaces of $\widehat{X}$ to arcs.  The covering map $\widehat{X}\to X$ induces a combinatorial map of underlying graphs $\widehat{\Xi}\to \Xi$, which sends vertices to vertices and edges to edges.

It remains to describe the attaching maps of the covering space $\widehat{X}$.  Given an attaching map $\partial^{\pm}_e:X_e\to X_v$ of $X$ and a vertex space $\widehat{X}_{\hat{v}}$ of $\widehat{X}$, the disjoint union of the edge spaces $\{\widehat{X}_{\hat{e}}\}$ of $\widehat{X}$ that lie above $X_e$, together with the coproduct of their attaching maps, fits into a commutative diagram as follows.

\[\xymatrix{
     \coprod_{\hat{e}}\widehat{X}_{\hat{e}}\ar@{>}[r]^{\coprod_{\hat{e}}\partial^{\pm}_{\hat{e}}}\ar@{>}[d] & \widehat{X}_{\hat{v}}\ar@{>}[d] \\
     X_e\ar@{>}[r]^{\partial^{\pm}_e} & X_v
}\]
Here, $\hat{e}$ ranges over all the edges of $\widehat{\Xi}$ in the preimage of $e$.  The key observation is that this diagram is a pullback.  The restriction of the pullback of a continuous map and a covering map to a connected component is called an \emph{elevation}.   See \cite{wise_subgroup_2000,wilton_elementarily_2007,wilton_halls_2008} for further characterisations of elevations. Therefore, the attaching maps of $\widehat{X}$ are precisely the elevations of the attaching maps of $X$ to the vertex spaces of $\widehat{X}$.  In particular, the covering space $\widehat{X}$ and covering map $\widehat{X}\to X$ are determined by the restriction of the covering map to the vertex spaces and by the map of graphs $\widehat{\Xi}\to \Xi$.

There is a condition on a set of covering maps of the vertex spaces that determines whether or not they can be extended to a covering of $X$.  The \emph{degree} of an elevation is the conjugacy class of $\pi_1(\widehat{X}_{\hat{e}})$ as a subgroup of $\pi_1(X_e)$.

\begin{proposition}
Let $\{\widehat{X}_{\hat{v}}\to X_v\}$ be a set of covering maps of the vertex spaces $\{X_v\}$ of a graph of spaces $X$. These covering maps can be extended to a covering map $\widehat{X}\to X$, where the $\{\widehat{X}_{\hat{v}}\}$ are the vertex spaces of $\widehat{X}$ in the induced graph-of-spaces decomposition, if and only if the following condition holds.  For each edge $e$ of $\Xi$ with attaching maps $\partial^{\pm}_e:X_e\to X_{v_{\pm}}$, there is a degree-preserving bijection between the set of elevations of $\partial^+_e$ to $\coprod_{\hat{v}} \widehat{X}_{\hat{v}}$ and the set of elevations of $\partial^-_e$ to $\coprod_{\hat{v}} \widehat{X}_{\hat{v}}$.
\end{proposition}

More generally, any degree-preserving bijection between subsets of the sets of elevations of the edge maps of $X$ to $\coprod_{v'} X'_{v'}$ can be used to build a graph of spaces $X'$ with vertex spaces $\{X'_{v'}\}$ and a map $X'\to X$.  The resulting space $X'$ is called a \emph{pre-cover} of $X$.  The elevations to $X'$ of attaching maps of $X$ that are not attaching maps of $X'$ are called \emph{hanging} elevations.  The previous proposition can be generalised as follows.

\begin{proposition}\label{p: Gluing precovers}
A pre-cover $X'$ of a graph of spaces $X$ can be extended to a covering map $\widehat{X}\to X$, where every vertex space of $\widehat{X}$ is a vertex space of $X'$, if and only if the following condition holds.  For each edge $e$ of $\Xi$ with attaching maps $\partial^{\pm}_e:X_e\to X_{v_{\pm}}$, there is a degree-preserving bijection between the set of hanging elevations of $\partial^+_e$ to $X'$ and the set of hanging elevations of $\partial^-_e$ to $X'$.
\end{proposition}

It follows that pre-covers can always be extended to covers.

\begin{lemma}\label{l: Pre-covers extend to covers}
If $X'$ is a pre-cover of $X$ then the map $X'\to X$ can be extended to a covering map $\widehat{X}\to X$, where $X'$ is a sub-graph of spaces of $\widehat{X}$ and the inclusion $X'\hookrightarrow \widehat{X}$ induces an isomorphism on fundamental groups.
\end{lemma}
\begin{proof}
Let $\partial^{\pm}_{\hat{e}}:\widehat{X}_{\hat{e}}\to X'$ be a hanging elevation to $X'$ of an attaching map $\partial^{\pm}_e:X_e\to X_v$.  For each such elevation $\partial^{\pm}_{\hat{e}}$, let $Y_{\hat{e}}$ be the covering space of $X$ with fundamental group $\pi_1(\widehat{X}_{\hat{e}})$.  There is an edge space of $Y_{\hat{e}}$ with fundamental group $\pi_1(\widehat{X}_{\hat{e}})$.  Delete this edge space from $Y_{\hat{e}}$, and let $Z_{\hat{e}}$ be a component of the result.  Then $Z_{\hat{e}}$ is a pre-cover of $X$ with a unique hanging elevation, which we denote by $\partial^{\mp}_{\hat{e}}:\widehat{X}_{\hat{e}}\to Z_{\hat{e}}$.  The pre-cover of $X$ that consists of the disjoint union of $X'$ and all the $Z_{\hat{e}}$, where $\hat{e}$ ranges over all the hanging elevations to $X'$ of attaching maps of $X$, satisfies the hypotheses of the previous proposition.  The resulting covering space $\widehat{X}$ has the required properties by construction.
\end{proof}

As an immediate result, pre-covers define subgroups.

\begin{lemma}
If $X'$ is a connected pre-cover of a graph of spaces $X$ then the map $X'\to X$ induces a monomorphism at the level of fundamental groups.
\end{lemma}

A graph of spaces $X$ is called \emph{reduced} if no attaching map is a $\pi_1$-surjection.

\begin{lemma}
Suppose $X$ is a reduced graph of spaces.  If $X'$ is a pre-cover but not a cover of $X$ then $\pi_1(X')$ has infinite index in $\pi_1(X)$.
\end{lemma}
\begin{proof}
Consider the space $Y_{\hat{e}}$ constructed in the proof of Lemma \ref{l: Pre-covers extend to covers}, and let $\Upsilon$ be its underlying graph.  Because $\pi_1(Y_{\hat{e}})$ is equal to the fundamental group of one of its edge spaces, $\Upsilon$ is a tree.  Because $X$ is reduced, every vertex of $\Upsilon$ has valence greater than one.  Therefore, $\Upsilon$ is infinite.  It follows that there are points of $X$ with infinitely many pre-images in $Z_{\hat{e}}$, and hence $\widehat{X}\to X$ has infinite degree.
\end{proof}

\section{A variant of a theorem of Shenitzer}

In this section, we prove a theorem that describes when the fundamental group of a graph of groups with cyclic edge groups splits freely.

\begin{theorem}\label{t: Variant of Shenitzer's Theorem}
Let $\Gamma$ be finitely generated, and the fundamental group of a graph of groups with infinite cyclic edge groups.   Then $\Gamma$ is one-ended if and only if every vertex group is freely indecomposable relative to the incident edge groups.
\end{theorem}

This statement is similar to a theorem of Shenitzer \cite{shenitzer_decomposition_1955}, which says that an amalgam of two free groups along a cyclic group is free if and only if the amalgamating cyclic group is a free factor in one of the free groups---see \cite{louder_krull_2008} for a modern treatment and a generalisation.  See also \cite{diao_grushko_2005}.  In the case of doubles, Theorem \ref{t: Variant of Shenitzer's Theorem} was stated without proof in \cite{gordon_surface_2010}.  It is undoubtedly well known to experts, but we were unable to find a proof in the literature, so we give one here.

\begin{remark}
The hypothesis that the edge groups are cyclic cannot be removed.  Indeed, the free group of rank two can be written as an HNN extension of $F_3\cong\langle a,b,c\rangle$ that conjugates $\langle a,b\rangle$ to $\langle b,c\rangle$, but $F_3$ does not split freely relative to $\langle a,b\rangle$ and $\langle b,c\rangle$.
\end{remark}

One direction of the theorem is obvious: if some vertex group splits freely relative to the incident edge groups then $\Gamma$ also splits freely.

To prove the other direction, we will realise $\Gamma$ as the fundamental group of a certain complex of groups; see \cite{bridson_metric_1999} for a detailed discussion of Haefliger's theory of complexes of groups.  By hypothesis, we are given a splitting of $\Gamma$ with finitely generated vertex groups $\Gamma_v$ and infinite cyclic edge groups $\Gamma_e$; let $\Xi$ be the underlying graph of this splitting. Now let $T$ be the Bass--Serre tree of a non-trivial free splitting of $\Gamma$.  Because the edge stabilisers of $T$ are trivial, any finitely generated non-trivial subgroup $H$ of $\Gamma$ has a unique minimal invariant subtree of $T$, on which $H$ acts cocompactly.  Let $T_v$ be the minimal invariant subtree of the vertex group $\Gamma_v$, and write $X_v$ for the quotient graph of groups $T_v/\Gamma_v$.  Likewise, let $T_e$ be the minimal invariant subtree of the edge group $\Gamma_e$, which is either a point or a line; again, write $X_e$ for the quotient graph of groups $T_e/\Gamma_e$, which is either topologically a circle or a point labelled by $\Gamma_e$.  The graphs of groups $X_v$ and $X_e$ are finite.  If $e$ is incident at $v$ then there is a natural inclusion $T_e\to T_v$ which descends to a morphism of graphs of groups $X_e\to X_v$.

Using these morphisms $X_e\to X_v$ as attaching maps, we can build a complex of groups, indeed a graph of graphs of groups, $\mathcal{X}$, with fundamental group $\Gamma$, just as we built a graph of spaces at the beginning of Section \ref{s: Covering theory}.  That is, the vertex `spaces' of $\mathcal{X}$ are the graphs of groups $X_v$, the edge `spaces' are the cylinders $X_e\times [0,1]$, the attaching maps are given by the morphisms $X_e\to X_v$, and the underlying graph is $\Xi$.  The universal cover of $\mathcal{X}$, which we denote by $\widetilde{\mathcal{X}}$, is a tree of copies of trees $T_v$, glued along strips or arcs of the form $T_e\times I$.  The underlying tree of $\widetilde{\mathcal{X}}$, which is the Bass--Serre tree of the given cyclic splitting of $\Gamma$, we denote by $\mathcal{T}$.

There is a natural $\Gamma$-equivariant map $f:\widetilde{\mathcal{X}}\to T$ defined as follows: the restriction of $f$ to $T_v$ is the inclusion $T_v\hookrightarrow T$; on $T_e\times [-1,+1]$, $f$ is the composition of the projection to $T_e$ with the inclusion $T_e\hookrightarrow T$.  Let $t$ be the midpoint of an edge in $T$, let $\widetilde{Y}_t =f^{-1}(t)$, and let $Y_t$ be the image of $\widetilde{Y}_t$ in $\mathcal{X}$.  By construction, for any $v$ the intersection of $Y_t$ with $X_v$ is a finite union of points, and the intersection of $Y_t$ with $X_e\times[-1,+1]$ is a finite union of arcs of the form $*\times [-1,+1]$.  Therefore $Y_t$ is topologically a finite graph, with vertices in the vertex graphs of groups of $\mathcal{X}$ and edges in the edge cylinders of $\mathcal{X}$.

\begin{lemma}
The inclusion $Y_t\hookrightarrow \mathcal{X}$, when restricted to a component, induces an injection on fundamental groups.
\end{lemma}
\begin{proof}
Let $Z$ be a component of $Y_t$ and let $\widetilde{Z}$ be a component of the preimage of $Z$ in $\widetilde{\mathcal{X}}$.  The lemma follows from the fact that the composition $\widetilde{Z}\hookrightarrow \widetilde{\mathcal{X}}\to\mathcal{T}$ is injective.  Suppose not.  Then, because $\mathcal{T}$ is a tree, some pair of adjacent edges in $\widetilde{Z}$ map to the same edge in $\mathcal{T}$.  This implies that this pair of edges of $\widetilde{Z}$ is contained in a single edge strip $T_e\times [-1,+1]$ of $\widetilde{\mathcal{X}}$.  But the gluing maps map $T_e\to T_v$ are injective, so no two edges of $\widetilde{Z}$ contained in $T_e\times[-1,+1]$ are adjacent.
\end{proof}

\begin{proof}[Proof of Theorem \ref{t: Variant of Shenitzer's Theorem}]
The free splitting of $\Gamma$ represented by $T$ is assumed to be non-trivial, so there is an edge of $T$ in the image of $\mathcal{\widetilde{X}}$; let $t$ be the midpoint of this edge. Because the map $\widetilde{Y}_t\to T$ is $\Gamma$-equivariant and its image is a point, any loop in $Y_t$ stabilises an edge in $T$. But the splitting corresponding to $T$ has trivial edge groups, so $Y_t$ is a finite forest.  Therefore, there is a vertex graph of groups $X_v$ with an edge $\epsilon$ whose midpoint is either an isolated vertex or a leaf in $Y_t$.  In the former case, deleting the midpoint of $\epsilon$ splits $\Gamma_v$ freely, relative to its incident edge groups.  In the latter case, there is a unique edge $e$ incident at $v$ with $\epsilon$ in the image of $X_e$, and only one edge of $X_e$ maps to $\epsilon$.  It follows that $\epsilon$ is non-separating in $X_v$---if it were separating, then every loop in $X_v$ would cross it an even number of times.  Therefore, $\epsilon$ corresponds to a basis element of $\Gamma_v$; because $\epsilon$ is crossed once by $X_e$ and is not crossed by the other incident edge spaces, there is a splitting of $\Gamma_v$ relative to the incident edge groups.
\end{proof}

\section{Whitehead graphs}

The aim of this section is to prove Theorem \ref{t: Local theorem}.  To do this, we first need to recall something about Whitehead's algorithm.   We will adopt the discs-in-handlebodies point of view---see \cite{berge_documentation_????} and \cite{manning_virtually_2010} for clear accounts.  For the combinatorial perspective on Whitehead's algorithm, see for instance \cite{lyndon_combinatorial_1977}.

Let $U$ be a handlebody of appropriate genus, and fix an identification $F\cong\pi_1(U)$.  The conjugacy classes of the elements of a multiword $\w$ naturally correspond to a 1-dimensional submanifold $N\subseteq U$, determined up to homotopy, in which each component is essential.   Any choice of basis $B$ for $F$ naturally corresponds to a maximal family of separating discs $\mathcal{D}\subseteq U$; when $U$ is cut along $\mathcal{D}$, the result is a 3-ball; the preimage of each disc is a pair of discs on the boundary of the ball, and the preimage of $N$ is a union of arcs in the ball, with their endpoints in the discs.  Crushing the discs to points, the resulting 1-complex embedded in the ball is the \emph{Whitehead graph} of $\w$, with respect to $B$.  It is convenient to remember the pairing on the vertices.  We will denote the Whitehead graph of a multiword $\w$ with respect to $B$ by $W_B(\w)$. We will often suppress all mention of $B$ when it causes no confusion.   We will call a multiword $\w$ \emph{minimal} if the basis $B$ is chosen to minimise its length.

\begin{figure}[htp]
\begin{center}
\includegraphics[width=0.6\textwidth]{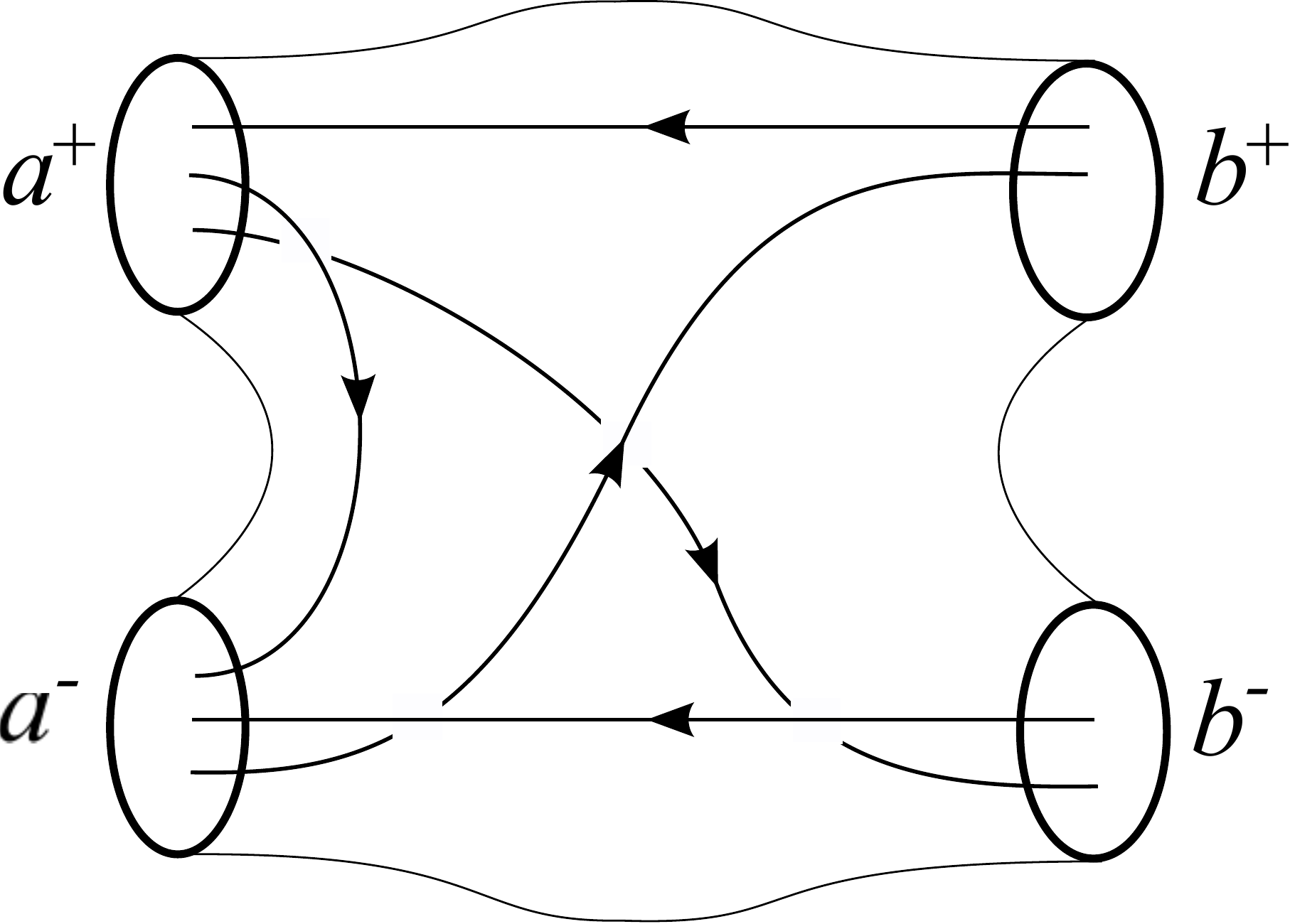}
\caption{The word $ba^{-1}b^{-1}a^2$  realised as a submanifold of a handlebody, which is cut along the discs that represent the standard generating set.}
\end{center}
\end{figure}

If there is a pair of vertices $\{v^+,v^-\}$ of $W_B(w)$ separated by a set $\mathcal{E}$ of edges of $W_B(\w)$, and the cardinality of $\mathcal{E}$ is strictly less than the valence of $v^+$ (which equals the valence of $v^-$) then a \emph{Whitehead move} can be applied, which yields a new basis $B'$ such that $W_{B'}(\w)$ has fewer edges than $W_B(\w)$.

For our purposes, Whitehead graphs will be useful for recognising free and cyclic splittings of $(F,[\w])$.  The following lemma is standard---see, for instance, \cite[Theorem 4.1]{cashen_line_2010}.

\begin{lemma}\label{l: Cut vertices}
If $W(\w)$ is disconnected then $(F,\w)$ is freely decomposable. Conversely, if $(F,\w)$ is freely decomposable then $W(\w)$ is either disconnected or contains a separating vertex.
\end{lemma}

Note that if $W(\w)$ contains a separating vertex then $\w$ is not minimal.  Next, we need a condition to recognise rigid pairs.  The following result follows from work of Cashen \cite{cashen_splitting_2010}, but we shall give a more direct proof here.

\begin{lemma}\label{l: Separating pair of edges}
If $W(\w)$ contains a separating pair of edges then $(F,[\w])$ is either a thrice-punctured sphere or is not rigid.
\end{lemma}
\begin{proof}
As above, we realise $\w$ by an embedded 1-dimensional submanifold in a handlebody $U$.  We will abuse notation and denote by $w_i$ the component of the submanifold corresponding to the element $w_i\in\w$.   Let $w_i,w_j\in \w$ be the (not necessarily distinct) components of $\w$ that contain the separating pair of edges.   By hypothesis, there exists a separating disc $D$ properly embedded in $U$ such that $w_i\cup w_j$ intersects $D$ in exactly two points and every other $w_k$ is disjoint from $D$.  Let $V$ be a small closed neighbourhood of $D\cup w_i\cup w_j$, let $\Sigma$ be the boundary of $V$ in $U$ and let $V'$ be the complement of the interior of $V$.

Fix a basepoint in $\Sigma$.  There are various cases to consider, according to whether or not $w_i$ and $w_j$ are distinct and whether or not $\Sigma$ is connected.  We will assume that they are distinct and that $\Sigma$ is connected, and leave the remaining cases to the reader.

By construction, $\pi_1(V)\cong \langle w_i\rangle *\langle w_j\rangle$ (choosing representatives of conjugacy classes appropriately).  The surface $\Sigma$ is homeomorphic to a twice-punctured torus, and its fundamental group can be presented as
\[
\pi_1(\Sigma)\cong\langle  a,b,d_1,d_2\mid [a,b]d_1d_2\rangle
\]
where $a=w_iw_j$, $b$ is a meridian about $w_i$, and $d_1$ and $d_2$ are freely homotopic to the boundary components of $\Sigma$.   The Seifert--van Kampen Theorem implies that $F\cong \pi_1(U)$ is a pushout.

\[\xymatrix{
     \pi_1(\Sigma)\ar@{>}[r]\ar@{>}[d] & \pi_1(V)\ar@{>}[d] \\
    \pi_1(V')\ar@{>}[r] & \pi_1(U)
}\]
Factoring out the kernels of the maps to $\pi_1(U)$, we obtain a decomposition of $\pi_1(U)$ as an amalgamated free product.  Because $b$, $d_1$ and $d_2$ are all null-homotopic in $U$, whereas $a$ survives, the edge group is cyclic.

If this cyclic splitting is trivial then necessarily $w_i$ and $w_j$ generate $F$, and then either $(F,[\w])$ is freely decomposable or
\[
[\w]=[\{w_i,w_j,w_iw_j\}]~,
\]
in which case $(F,[\w])$ is a thrice-punctured sphere.
\end{proof}

\begin{lemma}\label{l: No separating vertex-edge pair}
If $(F,[\w])$ is rigid and $W(\w)$ contains a vertex $v$ and an edge $e$ such that $W(\w)\smallsetminus \{v,e\}$ is disconnected then either $(F,[\w])$ is a thrice-punctured sphere or $\w$ is not minimal.
\end{lemma}
\begin{proof}
Let $A,B$ be the components of $W(\w)\smallsetminus \{v,e\}$.  Without loss of generality, assume that $u$, the pair of $v$, is contained in $A$.  Let $E_A$ and $E_B$ be the sets of edges that joins $v$ to $A$ and $B$ respectively.  Then $E_A\cup\{e\}$ is a set of edges that separates $v$ from $u$, so one can perform a Whitehead move, and hence $\w$ is not minimal, unless the valence of $v$ is equal to the cardinality of $E_A\cup\{e\}$.  But this is true if and only if $E_B$ is a single edge $f$, in which case $\{e,f\}$ is a separating pair of edges.
\end{proof}

These conditions are certainly not exhaustive.  There are cyclically decomposable pairs for which the minimal Whitehead graph does not contain a cut pair of edges.

Next, we need to understand the Whitehead graphs of pullbacks.  The key operation is \emph{splicing}, which was introduced by Manning \cite{manning_virtually_2010}.

\begin{definition}
Let $A,B$ be finite graphs.  Let $u$ be a vertex of $A$ and $v$ a vertex of $B$, and assume that the valences of $u$ and $v$ are equal.  Fix a bijection $f$ between the edges of $A$ incident at $u$ and the edges of $B$ incident at $v$.  Construct a new graph $C$ as follows: delete the vertices $u$ and $v$ from $A$ and $B$, leaving the incident edges `hanging'; then glue the resulting hanging edges of $A$ and $B$ together according to the bijection $f$.  The graph $C$ is said to be obtained from $A$ and $B$ by \emph{splicing}.
\end{definition}

The construction of the Whitehead graph of a pullback can be summarised in the following lemma.  An example is illustrated in Figure \ref{fig: Splice}.

\begin{lemma}[Manning \cite{manning_virtually_2010}]
If $\widehat{F}\subseteq F$ is a subgroup of finite index and $\hat{\w}$ is the pullback of $\w$ to $\widehat{F}$ then $W(\hat{\w})$ is obtained by splicing $|F:\widehat{F}|$ copies of $W(\w)$.  (The choices made when splicing correspond exactly to the choice of basis for $\widehat{F}$.) 
\end{lemma}

\begin{figure}[htp]
\begin{center}
\includegraphics[width=0.8\textwidth]{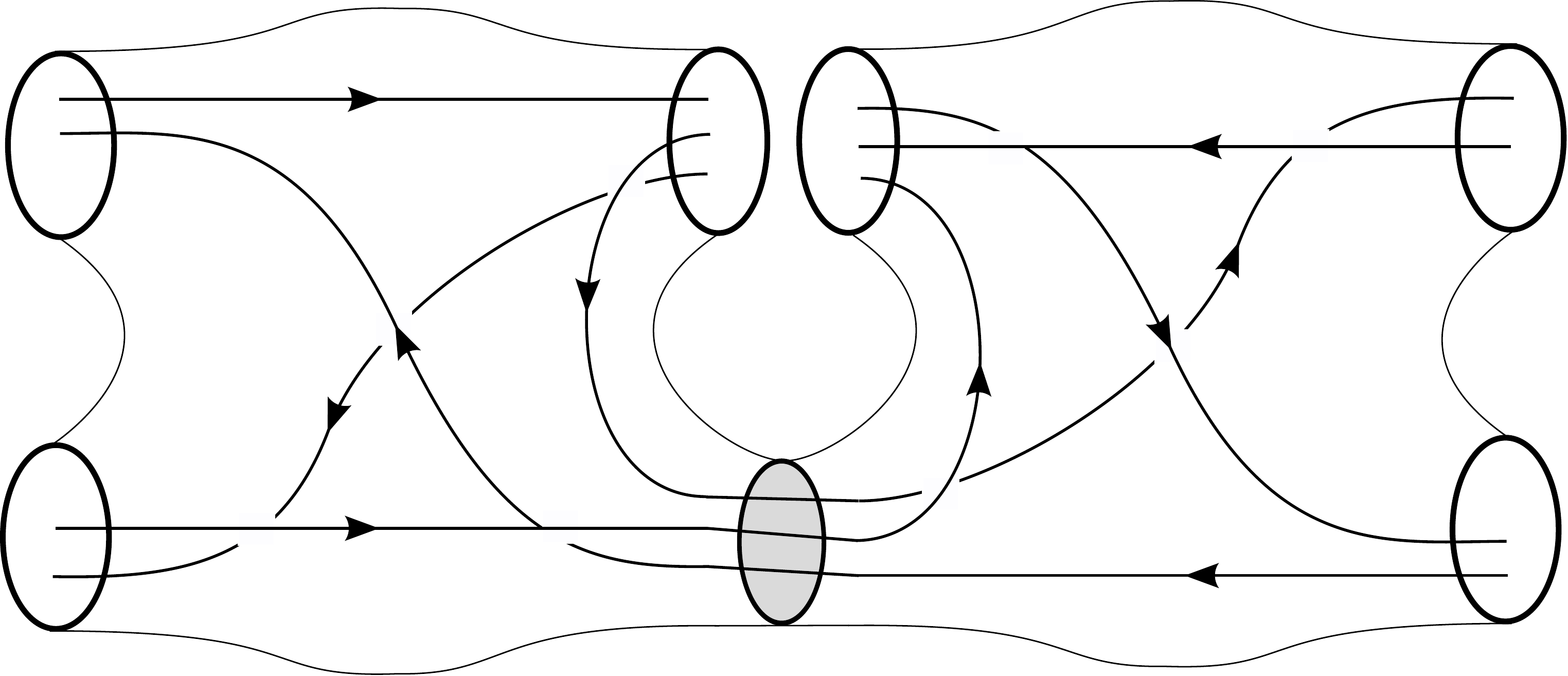}
\caption{The lift of $ba^{-1}b^{-1}a^2$  to a cover of degree two.  Deleting the grey disc corresponds to splicing.}\label{fig: Splice}  
\end{center}
\end{figure}

\begin{remark}
The result of splicing two connected graphs without cut vertices is another connected graph without cut vertices.  The proof of this is left as an exercise to the reader.
\end{remark}

  The final piece of the argument is the notion of a clean cover (or, equivalently, a clean subgroup), used extensively by Wise.

\begin{definition}\label{d: Clean}
Fix a basis $B$ for $F$ and consider a pair $(F,[\w])$.  Let $X$ be the rose graph, oriented and labelled, with the identification $F\cong\pi_1(X)$ determined by the choice of basis $B$.  Consider a subgroup $\widehat{F}$ of finite index in $F$, and the corresponding covering map $\widehat{X}\to X$.  The subgroup $\widehat{F}$ is called \emph{clean} if the pullback map
\[
\hat{\w}:\coprod_i S^1\to\widehat{X}
\]
is injective on each connected component.
\end{definition}

In the handlebody picture, the subgroup $\widehat{F}$ is clean if and only if every component of the preimage of $\w$ intersects each ball in at most one arc.  See Figure \ref{fig: Clean}.

\begin{figure}[htp]
\begin{center}
\includegraphics[width=0.5\textwidth]{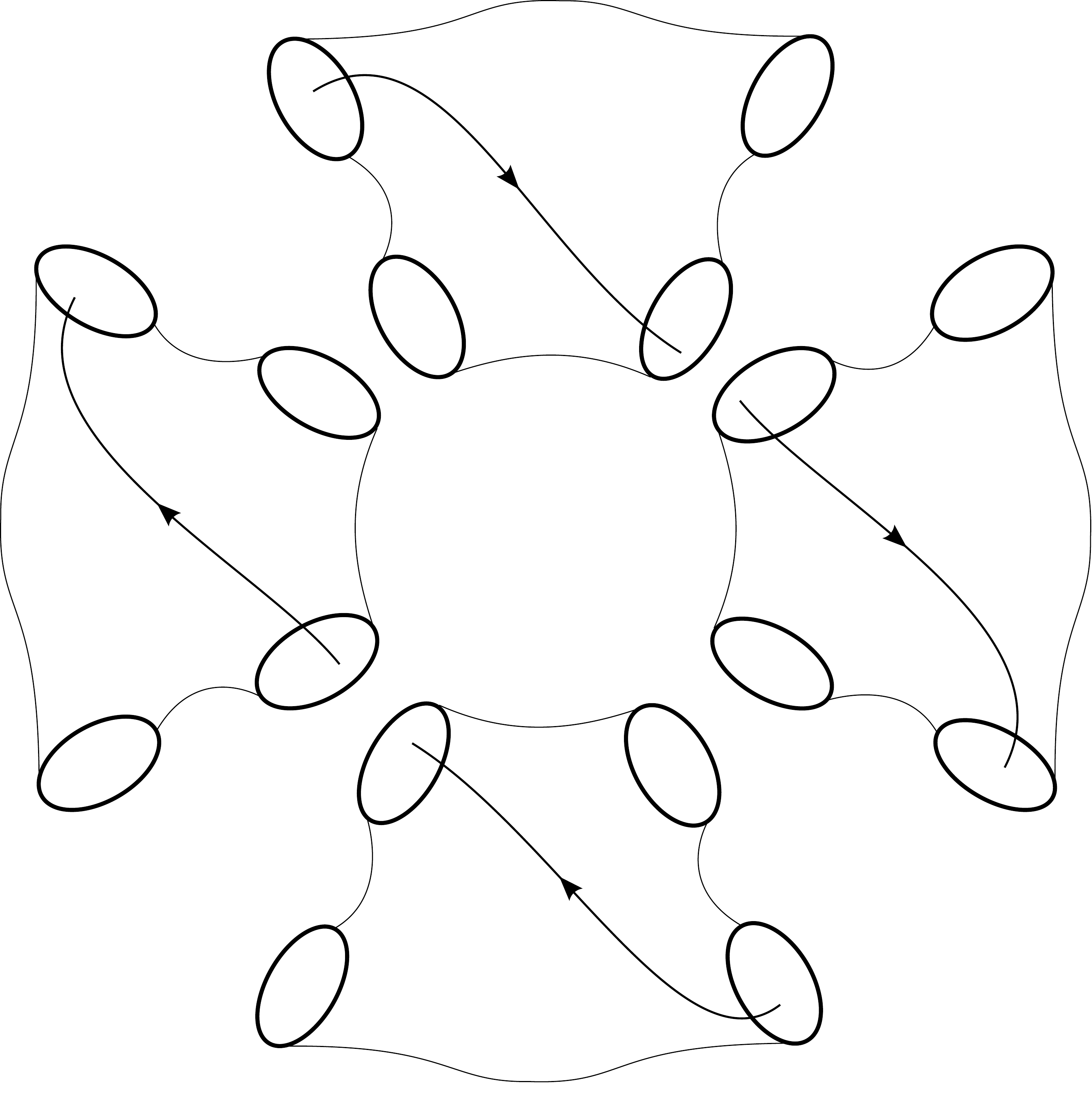}
\caption{One component of the pullback of a word to a clean cover.}  \label{fig: Clean}
\end{center}
\end{figure}

A standard application of Marshall Hall's Theorem shows that clean subgroups are plentiful.

\begin{lemma}
If $\overline{F}$ is any subgroup of finite index in $F$ then there is a clean subgroup $\widehat{F}\subseteq \overline{F}$.
\end{lemma}
\begin{proof}
For each element $w_i$ of $\w$, Marshall Hall's Theorem provides a finite-index subgroup $F_i\subseteq F$ such that, if $X_i\to X$ is the corresponding finite-sheeted cover, the map $w_i:S^1\to X$ lifts to an embedding $S^1\hookrightarrow X_i$.  The normal core of the subgroup
\[
\overline{F}\cap\bigcap_i F_i
\]
is the required clean subgroup of $F$.
\end{proof}

\begin{proof}[Proof of Theorem \ref{t: Local theorem}]
Choose a basis $B$ for $F$ so that $\w$ is minimal.  By Lemma \ref{l: No separating vertex-edge pair} and the hypothesis that $(F,\w)$ is rigid and not a thrice-punctured sphere, the Whitehead graph $W(\w)$ does not contain a separating vertex-edge pair.  As described above, the Whitehead graph $W(\hat{\w})$ is obtained by splicing together various copies $W_j$ of $W(\w)$.  Because $\widehat{F}$ is clean, $\hat{w}_i$ intersects each $W_j$ in at most one edge.  For each $j$, let $W'_j= W_j\smallsetminus \hat{w}_i$; note that each $W'_j$ is connected and does not have a cut vertex.  The Whitehead graph $W(\hat{\w}\smallsetminus\{\hat{w}_i\})$ is obtained by splicing together the $W'_j$, and therefore is also connected with no cut vertices.  In particular, $(\widehat{F},[\w\smallsetminus\{\hat{w}_i\}])$ is one-ended by Lemma \ref{l: Cut vertices}. 
\end{proof}

\section{The JSJ decomposition}

To prove the main theorem, we will use the (cyclic) JSJ decomposition of $\Gamma$.  JSJ decompositions of groups come in many different versions; \cite{guirardel_jsj_2009} contains a useful summary and a unifying perspective.  To be specific, we will use Bowditch's JSJ decomposition \cite{bowditch_cut_1998}, which is only defined for hyperbolic groups but has the advantage of being canonical.   We summarise its properties in the case of interest to us in the following result.

\begin{theorem}[Bowditch \cite{bowditch_cut_1998}]
Let $G$ be any one-ended, torsion-free, hyperbolic group.   There is a reduced graph of groups, with fundamental group $G$, with the following properties.
\begin{enumerate}
\item Each edge group is cyclic.
\item Each vertex group is of one of three sorts.
\begin{enumerate}
\item Cyclic vertex groups are cyclic.
\item Surface vertex groups are isomorphic to the fundamental groups of compact surfaces.  The incident edge groups are identified with (powers of) boundary components.
\item Rigid vertex groups do not split over any cyclic subgroup relative to the incident edge groups.
\end{enumerate}
Furthermore, one end of every edge adjoins a cyclic vertex.
\end{enumerate}
This graph of groups is called the \emph{JSJ decomposition} of $G$. 
\end{theorem}

Of course, in the above definition there is some ambiguity, as any vertex corresponding to a thrice-punctured sphere can be thought of as either of surface or rigid type.  We shall always think of them as being of surface type.

\begin{remark}
If $\Gamma$ is the fundamental group of a graph of free groups with cyclic edge groups then the JSJ decomposition of $\Gamma$ is trivial if and only if $\Gamma$ is the fundamental group of a surface.
\end{remark}

First, we shall deal with the case in which there are no rigid vertices.  The proof of the following lemma is similar to the proof of \cite[Lemma 22]{gordon_surface_2010}. 

\begin{lemma}\label{l: No rigid vertices}
Suppose the JSJ decomposition of $\Gamma$ has no rigid vertices.  Then $\Gamma$ has a surface subgroup.  In particular, $\Gamma$ either has a one-ended subgroup of infinite index or $\Gamma$ is a surface group.
\end{lemma}
\begin{proof}
By \cite[Theorem 4.18]{wise_subgroup_2000}, we can pass to a finite-index subgroup $\Gamma_0$ of $\Gamma$ and assume that every attaching map identifies a cyclic vertex group with a boundary component of a surface component.  It is easy to thicken this picture to see that $\Gamma_0$ is the fundamental group  of a 3-manifold with boundary, $M$.  Standard gluing arguments show that $M$ is aspherical, and it follows that $\chi(\partial M)\leq 0$.  By Stallings's Ends Theorem, $\Gamma_0$ is one-ended, and so $\partial M$ is incompressible in $M$ by the Loop Theorem and Dehn's Lemma.  Therefore $\pi_1(\partial M)$ is a surface subgroup of $\Gamma$.
\end{proof}

\begin{proof}[Proof of Theorem \ref{t: Main theorem}]
Realise the JSJ decomposition of $\Gamma$ by a graph of spaces $X$. By Lemma \ref{l: No rigid vertices}, we may assume that some vertex $v$ of $X$ is rigid, and not a thrice-punctured sphere.  That is, if $F=\pi_1(X_v)$ and $[\w]$ is the peripheral structure induced by the incident edge groups, then $(F,[\w])$ is rigid.  Let $\widehat{F}$ be a clean subgroup of finite index in $F$, equipped with the pullback peripheral structure $[\hat{\w}]$.    Because $\Gamma$ is subgroup separable \cite{wise_subgroup_2000}, there is a finite-sheeted covering space $\widehat{X}$ of $X$, with a vertex space $\widehat{X}_{\hat{v}}$ covering $X_v$, such that $\pi_1(\widehat{X}_{\hat{v}})=\widehat{F}$.  Let $\widehat{\Gamma}=\pi_1(\widehat{X})$ and let $\widehat{\Xi}$ be the underlying graph of $\widehat{X}$.  Let $X'$ be the pre-cover of $X$ that consists of the union of every vertex space of $\widehat{X}$ apart from $\widehat{X}_{\hat{v}}$, and  every edge cylinder of $\widehat{X}$ that does not adjoin $\widehat{X}_{\hat{v}}$.  Let $e_1,\ldots,e_m$ be the edges of $\widehat{\Xi}$ with one end adjoining $v$, and let $f_1,\ldots,f_n$ be the edges of $\Xi$ with both ends adjoining $v$.  We denote by $\partial^+_{e_i}$ the attaching map of $e_i$ that maps to $X_v$ and by $\partial^-_{e_i}$ the attaching map that maps to $X'$.   Both $\partial^+_{f_j}$ and $\partial^-_{f_j}$ map to $X_v$.

We will now define a new pre-cover $W$ of $\widehat{X}$ as follows.  Take $m$ copies of $\widehat{X}_{\hat{v}}$, denoted by $Y_i$, and take $2n$ further copies of $\widehat{X}_{\hat{v}}$, denoted by $Y^\pm_j$.  Consider the following subset $\mathcal{E}$ of the set of all elevations of edge maps of $X$ to $Y=\coprod_i Y_i\sqcup\coprod_j Y^+_j\sqcup\coprod_j Y^-_j$: $\mathcal{E}$ consists of all such elevations, except for the copy of $\partial^+_{e_i}$ that maps to $Y_i$, the copy of $\partial^+_{f_j}$ that maps to $Y^+_j$ and the copy of $\partial^-_{f_j}$ that maps to $Y^-_j$.  Each attaching map incident at $v$ has exactly $m+2n-1$ elevations in $\mathcal{E}$, and each of these elevations is a lift (meaning that it is of maximal possible degree).

Now take $m+2n-1$ copies of $X'$, denoted denote by $Z_k$, and let $Z=\coprod_k Z_k$.  The complete set of elevations of edge maps of $\widehat{X}$ to $Z$, which we denote by $\mathcal{F}$, consists of exactly $m+2n-1$ elevations of each $\partial^-_{e_i}$, and again each of these elevations is a lift.  Therefore, $Y\sqcup Z$, together with the elevations $\mathcal{E}\cup \mathcal{F}$, satisfies the hypotheses of Proposition \ref{p: Gluing precovers}, and can be completed to a pre-cover $W$ of $\widehat{X}$.  If $W$ is disconnected, replace it with one of its connected components.  Each vertex group of $W$ is freely indecomposable relative to its incident edge groups, so $\pi_1(W)$ is one-ended.  The $Y_i$ have hanging elevations of edge groups of $\widehat{X}$, so $W$ is not a cover of $\widehat{X}$.  Therefore $\pi_1(W)$ is a subgroup of infinite index in $\widehat{\Gamma}$, and hence in $\Gamma$ as required.
\end{proof}

Finally, we prove Corollary \ref{c: Minimal commensurability classes}.  One possible proof is identical to the proof of Theorem  \ref{t: Main theorem}, but uses a relative version of the JSJ decomposition, such as that provided by \cite{cashen_splitting_2010}.  The proof given here deduces Corollary \ref{c: Minimal commensurability classes} directly from Theorem \ref{t: Main theorem} by considering doubles.

Given a pair $(F,[\w])$, the corresponding \emph{double} $D(F,[\w])$ is the fundamental group of a graph of groups with two vertices and $\#[\w]$ edges; each vertex is labelled by a copy of $F$, each edge group is an element of $[\w]$, and the attaching maps are the natural inclusions.  Theorem \ref{t: Variant of Shenitzer's Theorem} implies that the pair $(F,[\w])$ is one-ended if and only if the double $D(F,[\w])$ is one-ended.  Clearly, $D(F,[\w])$ is a surface group if and only the pair $(F,[\w])$ can be realised as $(\pi_1\Sigma,\partial\Sigma)$ for some compact surface $\Sigma$.

\begin{proof}[Proof of Corollary \ref{c: Minimal commensurability classes}]
It is clear that a pair of the form $(\pi_1\Sigma,\partial\Sigma)$ is minimal.  Conversely, suppose that $(F,[\w])$ cannot be realised by a surface.  Then $D(F,[\w])$ is not a surface group and so, by Theorem \ref{t: Main theorem}, has a finitely generated, one-ended subgroup $H$ of infinite index.  Let $X$ be a graph of spaces that realises the given decomposition of $D(F,[\w])$, and let $X^H$ be the covering space corresponding to $H$.  Because $H$ is finitely generated, there is a finite subgraph of spaces $X'\subseteq X^H$ with fundamental group $H$; $X'$ is naturally a pre-cover of $X$.    By Theorem \ref{t: Variant of Shenitzer's Theorem}, each vertex group of $X'$, equipped with the induced peripheral structure, is one-ended.  If every vertex space $X'_{v'}$ of $X'$ finitely covered the corresponding vertex space $X_v$ of $X$, and if the induced peripheral structure on $\pi_1 X'_{v'}$ were always the full pullback of the peripheral structure on $\pi_1 X_v$, then it would follow from the covering theory of graphs of spaces that $X'\to X$ was a finite-sheeted covering map, which would contradict the fact that $H$ is of infinite index in $D(F,[\w])$.  Therefore, there is a vertex space $\pi_1X'_{v'}$ that is either of infinite index in $F$ or such that the peripheral structure $[\uu]$ on $\pi_1 X'_{v'}$ induced by $X'$ is strictly contained in the pullback peripheral structure $[\hat{\w}]$.  Therefore, $(\pi_1 X'_{v'},[\uu])<(F,[\w])$ and so $(F,[\w])$ is not minimal. 
\end{proof}

\bibliographystyle{plain}
\bibliography{oneend.bib}

\end{document}